\documentclass[11pt]{amsart}
\usepackage{graphicx,amssymb,amsmath,amsthm}
\usepackage{enumerate}
\usepackage{dsfont}
\usepackage{booktabs}
\usepackage{mathrsfs}
\usepackage{verbatim}
\usepackage{epsfig}
\usepackage{lscape}
\usepackage{subfigure}
\usepackage[colorlinks, linkcolor=red, anchorcolor=blue, citecolor=green]{hyperref}

\usepackage{epstopdf}
\usepackage{color}
\usepackage{caption}
\usepackage{algorithm}
\usepackage{algpseudocode}

\newcommand{\abs}[1]{\lvert#1\rvert}
\newcommand{\absinn}[1]{\vert\langle {#1} \rangle\rvert}

\newcommand{\R}{{\mathbb R}}

\newcommand{\Sd}{{\mathbb S}^{d-1}}

\newcommand{\vx}{{\mathbf x}}
\newcommand{\vy}{{\mathbf y}}

\newcommand{\Z}{{\mathbb Z}}

\renewcommand{\omega}{\eta}

\newcommand{\RNum}[1]{\uppercase\expandafter{\romannumeral #1\relax}}

\newtheorem{definition}{Definition}[section]
\newtheorem{corollary}{Corollary}[section]
\newtheorem{prop}{Proposition}[section]
\newtheorem{theorem}{Theorem}[section]
\newtheorem{lemma}{Lemma}[section]
\newtheorem{conjecture}[definition]{Conjecture}

\newtheorem{remark}{Remark}[section]

\newtheorem{problem}{Problem}[section]
\date{}

\begin{document}
\baselineskip 14pt
\bibliographystyle{plain}
\title{Bounds on antipodal spherical designs with few angles}

\author{Zhiqiang Xu}
\thanks{Zhiqiang Xu was supported  by Beijing Natural Science
Foundation (Z180002) and by NSFC grant (11688101).}

\address{LSEC, Inst.~Comp.~Math., Academy of
Mathematics and System Science,  Chinese Academy of Sciences, Beijing, 100091, China
\newline
School of Mathematical Sciences, University of Chinese Academy of Sciences, Beijing 100049, China}
\email{xuzq@lsec.cc.ac.cn}

\author{Zili Xu}
\address{ LSEC, Inst.~Comp.~Math., Academy of
Mathematics and System Science,  Chinese Academy of Sciences, Beijing, 100091, China
\newline
School of Mathematical Sciences, University of Chinese Academy of Sciences, Beijing 100049, China}
\email{xuzili@lsec.cc.ac.cn}

\author{Wei-Hsuan Yu}
\address{ Mathematics Department, National Central University, Taoyuan, Taiwan}
\email{u690604@gmail.com}

\begin{abstract}

A finite subset $X$ on the unit  sphere $\Sd$ is called an $s$-distance  set with
strength $t$ if its angle set $A(X):=\{\langle \vx,\vy\rangle : \vx,\vy\in X,
\vx\neq\vy \}$ has size $s$, and $X$ is a spherical $t$-design but not a spherical
$(t+1)$-design. In this paper, we consider to estimate the maximum size of such
antipodal set $X$ for small $s$.  Motivated by the method developed by Nozaki and
Suda \cite{Nozaki},  for each even integer $s\in[\frac{t+5}{2}, t+1]$ with $t\geq 3$,
we improve the best known upper bound of Delsarte, Goethals and Seidel
\cite{Delsarte}. We next focus on two special cases: $s=3,\ t=3$ and $s=4,\ t=5$.
Estimating the size of $X$ for these two cases is equivalent to estimating the size
of real equiangular tight frames (ETFs) and Levenstein-equality packings,
respectively. We improve the previous estimate on the size of real ETFs and
Levenstein-equality packings. This in turn gives an upper bound on $\abs{X}$ when
$s=3,\ t=3$ and $s=4,\ t=5$, respectively.

\end{abstract}

\maketitle

\section{Introduction}
\subsection{Spherical designs with few angles}
A finite set $X\subset \Sd$ is called   an $s$-distance set if its angle  set
$A(X):=\{\langle \vx,\vy\rangle : \vx,\vy\in X, \vx\neq \vy \}$ contains $s$ distinct
values, and we say  $X$ has strength $t$ if $t$ is the largest integer such that $X$
is a spherical $t$-design.  We say that a finite set $X\subset\Sd$ is  a spherical
$t$-design if the following equality
\begin{equation*}
\int_{\Sd}f(\vx)d\mu_d(\vx)=\frac{1}{|X|}\sum_{\vx\in X}^{}f(\vx)
\end{equation*}
holds for any polynomial $f$ of degree at most $t$   (see \cite{Delsarte}). Here,
$\mu_d$ is the Lebesgue measure on $\Sd$ normalized by $\mu_d(\Sd)=1$. In this paper
we focus on the following problem which originally arises in design theory:
\begin{problem}\label{pr:1}
Given $s,t\in \Z_+$, what is the maximum size of an $s$-distance  set $X\subset\Sd$ with strength $t$?
\end{problem}
Spherical designs with few angles usually display beautiful  symmetry and optimality
\cite{Cohn,Bilyk,Hardin}, e.g., the universal optimality of the $600$-cell on
$\mathbb{S}^3$ \cite{Cohn}, which have been studied for several decades
\cite{Delsarte,Bannai2,Bannai3}. Estimating the size of these designs provides a
necessary condition on their existence. See \cite{Bannai,Nozaki} and the references
for the recent work.

In this paper we  devote our attention to the antipodal  case of this problem, i.e.,
$X=-X$. We aim to bound the size of antipodal $s$-distance sets in $\Sd$ with
strength $t$. Recall that the strength of an antipodal set must  be an odd integer
\cite[Theorem 5.2]{Delsarte}. According to \cite[Theorem 6.8]{Delsarte}, we always
have
\begin{equation}\label{absolute1}
|X|\leq  2\binom{d+s-2}{s-1}\quad \text{and}\quad 2s\geq t+1
\end{equation}
provided  $X\subset\Sd$ is an antipodal $s$-distance set  with strength $t$. The
upper bound in
 (\ref{absolute1}) is called the Delsarte-Goethals-Seidel bound for an antipodal
spherical $s$-distance set. Furthermore, the equality in (\ref{absolute1}) holds  if
and only if the $s$-distance set $X$ forms a tight spherical $(2s-1)$-design, i.e.,
$t+1=2s$. In this paper we will focus on estimating the size of $X$ when $2s$ is
slightly greater than $t+1$.

\subsection{The optimal line packing problem}
 It is particularly interesting to consider two special
cases, i.e., $s=3,\ t=3$ and $s=4,\ t=5$. These two cases are closely related to the
optimal line packing problem, which aims to find a finite set
$\Phi=\{\boldsymbol\varphi_i \}_{i=1}^{n}\subset\Sd$ with fixed size $n>d$ and the
minimal coherence $\mu(\Phi):=\max\limits_{i\neq j}|\langle
\boldsymbol\varphi_i,\boldsymbol\varphi_j\rangle|$ (see
\cite{Conway,Fickus2,Haas,Jasper}). The followings are two well-known lower bounds on
the coherence:
\begin{subequations}\label{welchleven}
\begin{align}
\mu(\Phi)&\,\,\geq\,\, \sqrt{\frac{n-d}{d(n-1)}},\ \ \ \quad \quad{\rm if}\ n> d, \label{welch}\\
\mu(\Phi)&\,\,\geq\,\,  \sqrt{\frac{3n-d(d+2)}{(d+2)(n-d)}},\ {\rm if}\ n>\frac{d(d+1)}{2}. \label{leven}
\end{align}
\end{subequations}
The (\ref{welch}) is called the Welch bound \cite{Welch} and the (\ref{leven}) is
called the Levenstein bound \cite{leven,leven2}. It is well known that the equality
in (\ref{welch}) occurs when $\Phi\cup-\Phi$ forms an antipodal $3$-distance
$3$-strength set or an antipodal $3$-distance $5$-strength set with size $d(d+1)$
\cite[Example 8.3]{Delsarte}; and the equality in (\ref{leven}) occurs when
$\Phi\cup-\Phi$ forms an antipodal $4$-distance $5$-strength set or an antipodal
$4$-distance $7$-strength set with size $\frac{d(d+1)(d+2)}{3}$ \cite[Example
8.4]{Delsarte}. Hence, estimating the size of the antipodal $3$-distance $3$-strength
sets and of the antipodal $4$-distance $5$-strength sets is  helpful  to know the
existence of these two kinds of optimal  packings. In the context of frame theory, a
set achieving the Welch bound in (\ref{welch}) is  known as a real equiangular tight
frame (ETF). Hence, bounding the size of an antipodal $3$-distance set with strength
$3$ is equivalent to bounding the size of a real ETF whose size is strictly smaller
than $\frac{d(d+1)}{2}$. This is particularly interesting since the existence of real
ETFs is a long-standing open problem for most pairs $(d,n)$ \cite{ETF,ETF2}. For the
nontrivial case where $n>d+1>2$, an ETF may exist only if its size $n$ satisfies the
Gerzon bound \cite{lemmens,ETF,Fickus2}:
\begin{equation}\label{Gerzon}
d+\frac12+\sqrt{2d+\frac14}\leq n\leq \frac{d(d+1)}{2}.
\end{equation}

\subsection{Related work}

We overview the known upper bounds on antipodal  spherical designs with few
angles. Let $X\subset\Sd$ be an antipodal $s$-distance set with strength $t$. As said
before,  $t$ must be an odd integer \cite[Theorem 5.2]{Delsarte}. Set
\begin{equation}\label{hi}
h_0:=1, \ h_1:=d,\ h_k:=\binom{d+k-1}{k}-\binom{d+k-3}{k-2},\,\, k\geq 2
\end{equation}
and
\begin{equation}\label{deltas}
\delta_s:=\left\{\begin{array}{ll} 0,  & \text{if $s$ is even,}\\
1, & \text{if $s$ is odd.}\end{array}\right.
\end{equation}
For each odd integer $t\in[s-\delta_s-1,2s-2\delta_s-3]$,  Nozaki and Suda in \cite[Corollary 3.7]{Nozaki}  derived a
new upper bound on $|X|$:
\begin{equation}\label{odds}
|X|\leq 2\binom{d+s-\delta_s-1}{s-\delta_s}-2h_{t-s+\delta_s+1}.
\end{equation}
If $s\geq 3$ is odd and $t\in[s-2,2s-5]$, it is easy to see that the bound in
(\ref{odds}) lowers the Delsarte-Goethals-Seidel bound by $2h_{t-s+2}$. If $s\geq 2$
is even and $t\in[s-1,2s-3]$, the upper  bound in (\ref{odds}) becomes
$2\binom{d+s-1}{s}-2h_{t-s+1}$. When $s$ is fixed,  a simple calculation shows that
$2\binom{d+s-1}{s}-2h_{t-s+1}=\Theta(d^s)$  while the Delsarte-Goethals-Seidel bound
in (\ref{absolute1}) is $\Theta(d^{s-1})$. Hence, if $s$ is an fixed even integer and
$d$ is large enough, the upper bound in (\ref{odds}) is larger  than the
Delsarte-Goethals-Seidel bound.


\subsection{Our contributions}

Assume that $X\subset\Sd$ is an antipodal $s$-distance set  with strength $t$. The
aim of this paper is to present a better upper bound on  $\abs{X}$.

\subsubsection{The general case}

Motivated by the methods developed in \cite{Nozaki}, we present an upper bound for
$\abs{X}$ which lowers the Delsarte-Goethals-Seidel bound  when
$s\in[\frac{t+5}{2},t+1]$ is an even integer and $t\geq 3$.

\begin{theorem}\label{general}
Let $d\geq2$ be an integer. Assume that $X\subset\Sd$ is an antipodal $s$-distance
set with strength $t \geq 3$,  where  $s\in[\frac{t+5}{2},t+1]$ is an even integer.
Then, we have
\begin{equation}\label{generalcase}
|X|\,\,\leq\,\, 2\binom{d+s-2}{s-1}-2h_{t-s+2},
\end{equation}
where $h_k$ is defined in (\ref{hi}) for each $k\geq 0$.
\end{theorem}

We next consider to estimate $|X|$ for the case when $s=\frac{t+3}{2}$. We mainly focus on two special cases : $s=3,\ t=3$ and $s=4,\ t=5$.

\subsubsection{The case: $s=3,\ t=3$}

As mentioned before, a set $\Phi$ is an ETF for $\R^d$ with size $n<\frac{d(d+1)}{2}$
if and only if $\Phi\cup-\Phi$ is an antipodal $3$-distance sets with strength $3$ \cite[Example 8.3]{Delsarte}.
Hence, we direct our attention to estimating the size of real ETFs. The following
theorem presents a necessary condition for  the size of real ETFs.

\begin{theorem}\label{coro1}
Let $d\geq5$ be an integer. Assume that $\Phi$ is an ETF for $\R^d$ with size
$n>d+1$. Then, we have either $n\in\{d+\frac12+\sqrt{2d+\frac14},\frac{d(d+1)}{2}\}$
or
\begin{equation}\label{g}
d+\frac12+\sqrt{3d+\frac14}\leq n\leq \frac{d(d+2)}{3}.
\end{equation}
\end{theorem}

 { Theorem \ref{coro1} improves the Gerzon bound (\ref{Gerzon}) when
$n\notin\{d+\frac12+\sqrt{2d+\frac14},\frac{d(d+1)}{2}\}$. To our knowledge, one only
finds two pairs $(d,n)$ for which ETFs exist and  achieve the size
$\frac{d(d+2)}{3}$: $(6,16)$ and $(22,176)$. The known configuration of these two
ETFs is a subset of ETFs with parameters $(7,28)$ and $(23,276)$, respectively (see
\cite[Remark 5.2]{Tremain} and \cite[Page 271]{Taylor}).
 }
Motivated by the observation, we present the following conjecture:
\begin{conjecture}\label{conj1}
Assume that $d\geq 5$. There exists an ETF with parameters
$(d+1,\frac{(d+1)(d+2)}{2})$ if and only if there exists an ETF with parameters
$(d,\frac{d(d+2)}{3})$.
\end{conjecture}

Recall that $X\subset \Sd$ is an antipodal $3$-distance set with strength $3$ if and
only if $X=\Phi\cup -\Phi$ where $\Phi$ is an ETF in $\R^d$ with size $n<\frac{d(d+1)}{2}$. A
simple observation is $\Phi\cap -\Phi=\emptyset$ if $\Phi$ is an ETF.
 Hence, we immediately obtain
an upper  bound  for the antipodal $3$-distance sets with strength $3$.

\begin{corollary}\label{wel}
Let $d\geq5$ be an integer. Assume $X\subset\Sd$ is an antipodal  $3$-distance set
with strength $3$. Then, we have either $|X|\in\{2d+2,2d+1+\sqrt{8d+1}\}$ or
\begin{equation*}
2d+1+\sqrt{12d+1}\leq|X|\leq\frac{2d(d+2)}{3}.
\end{equation*}
\end{corollary}
\begin{remark}
According to Delsarte-Goethals-Seidel bound in (\ref{absolute1}), $\abs{X}\leq
d(d+1)$ if  $X\subset\Sd$ is an antipodal  $3$-distance set with strength $3$.
 Corollary \ref{wel} lowers the bound to $\frac{2d(d+2)}{3}$.
\end{remark}
We next introduce  another result  on the existence of real ETFs. It is well known
that the existence of a real  ETF for $\R^d$ with size $n>d+1>2$ is equivalent to the
existence of a strongly regular graph with parameters
$(n-1,a,\frac{3a-n}{2},\frac{a}{2})$ \cite{ETF2,Waldron,ETF}, where
\begin{equation}\label{a}
a:=\frac{n}{2}-1+(1-\frac{n}{2d})\sqrt{\frac{d(n-1)}{n-d}}.
\end{equation}
Since every strongly regular graph satisfies the Krein  conditions (see Lemma
\ref{krein} for details), one is interested in whether the Krein conditions are
covered by the Gerzon bound (\ref{Gerzon}) or other known necessary conditions  (see
\cite{Waldron,ETF}). In Proposition \ref{pr:krein}  we will give a positive answer to
this question showing that the Krein conditions for strongly regular graphs with
parameters $(n-1,a,\frac{3a-n}{2},\frac{a}{2})$ are equivalent to the Gerzon bound
(\ref{Gerzon}).


\subsubsection{The case: $s=4,\ t=5$}

Finally, we consider to estimate the size of  antipodal $4$-distance sets with
strength $5$. {Recall that the Levenstein bound (\ref{leven}) is attained only if
$n>\frac{d(d+1)}{2}$ \cite[Theorem 6.13]{leven2}.} Moreover, $\Phi\subset\Sd$ is a
Levenstein-equality packing with size $n<\frac{d(d+1)(d+2)}{6}$ if and only if
$\Phi\cup-\Phi$ is an antipodal $4$-distance set with strength $5$ \cite[Example
8.4]{Delsarte}. Thus, we mainly focus on estimating the size of Levenstein-equality
packings.

We begin with providing an estimate on the size of  Levenstein-equality packings,
i.e., $\mu(\Phi)=\sqrt{\frac{3n-d(d+2)}{(d+2)(n-d)}}$.

\begin{theorem}\label{lev}
Let $d\geq4$ be an integer. Assume $\Phi\subset\Sd$ is  a Levenstein-equality packing
with size $n$.  { Then we have either $n=\frac{d(d+2)}{2}$ or
\begin{equation}\label{eq:nalpha1}
n\in \left\{\frac{d(d+2)(d-1+\alpha)}{3\alpha} : \alpha \in [2,\frac{2(d-1)(d+2)}{d+5}]\cap \Z\right\}\cap \Z.
\end{equation}}
Particularly, if $\Phi\subset\Sd$ is  a Levenstein-equality packing with size $n\notin\{
\frac{d(d+2)}{2},\frac{d(d+1)(d+2)}{6}\}$  then we have $n\in [\frac{d(d+3)}{2},\frac{d(d+2)^2}{9}]$.
\end{theorem}

\begin{remark}\label{re:tight}
Taking $\alpha=3$ in (\ref{eq:nalpha1}), we have $n=\frac{d(d+2)^2}{9}$. {To our
knowledge, so far, one only finds two pairs $(d,n)$ for which Levenstein-equality
packings achieve the size $\frac{d(d+2)^2}{9}$: $(7,63)$ and $(22,1408)$. The
corresponding packings come from a tight spherical $7$-design in $\R^8$ and
$\R^{23}$, respectively \cite[Page 619]{Munemasa}. We next explain the link between
tight spherical $7$-designs and Levenstein-equality packings. Assume that
$X\subset\R^{d+1}$ is a tight spherical $7$-design.
 According to
\cite[Theorem 8.2]{Delsarte},  for each $\vx\in X$, the derived code
$Z:=\{\vy\in X: \langle \vy,\vx \rangle=0 \}$ forms an antipodal 4-distance
$5$-strength set in $\R^d$ with the angle set $\{-1,0,\pm\sqrt{\frac{3}{d+5}}\}$.
Hence, choosing a point from each antipodal pair of points in $Z$ gives a
Levenstein-equality packing $\Phi\subset\Sd$. Combining
$\mu(\Phi)=\sqrt{\frac{3}{d+5}}$ with (\ref{leven}) we obtain that $\Phi$ has size
$\frac{d(d+2)^2}{9}$.}  Hence, if we have a tight spherical 7-design in $\R^{d+1}$,
we can obtain a Levenstein-equality packings in $\R^d$ with size
$\frac{d(d+2)^2}{9}$. We are interested in knowing whether each Levenstein-equality
packing with size $\frac{d(d+2)^2}{9}$ comes from  a tight spherical $7$-design.
\end{remark}

Remind that $X\subset \Sd$ is an antipodal $4$-distance set with strength $5$ if and
only if $X=\Phi\cup -\Phi$ where $\Phi$ is a Levenstein-equality packing in $\R^d$
with size $n<\frac{d(d+1)(d+2)}{6}$.  Based on Theorem \ref{lev}, we have the
following corollary:

\begin{corollary}\label{coro2}
Let $d\geq4$ be an integer. Assume $X\subset\Sd$ is an antipodal $4$-distance set
with strength $5$.  Then, we have either $|X|=d(d+2)$ or
\begin{equation}\label{s4t5}
\abs{X}\,\,\in\,\, \left\{\frac{2d(d+2)(d-1+\alpha)}{3\alpha} : \alpha \in [3,\frac{2(d-1)(d+2)}{d+5}]\cap \Z\right\}\cap \Z.
\end{equation}
\end{corollary}
\begin{remark}
According to (\ref{s4t5}), we obtain that  $\abs{X}\leq \frac{2d(d+2)^2}{9}$ if
 $X\subset\Sd$ is an antipodal $4$-distance set
with strength $5$.
 Recall that the Delsarte-Goethals-Seidel bound in (\ref{absolute1}) gives $|X|\leq
\frac{d(d+1)(d+2)}{3}$ when $s=4$, and that Nozaki-Suda bound in (\ref{odds}) gives
$|X|\leq \frac{1}{12}\cdot{(d+2)(d^3+4d^2-9d+12)}$ when $s=4, t=5$. A simple calculation shows
that
\begin{equation*}
\frac{2d(d+2)^2}{9}<\frac{d(d+1)(d+2)}{3}<\frac{(d+2)(d^3+4d^2-9d+12)}{12}, \quad d\geq 4.
\end{equation*}
Hence, the upper bound $\abs{X}\leq \frac{2d(d+2)^2}{9}$ improves both Nozaki-Suda bound and
the Delsarte-Goethals-Seidel bound when $d\geq 4$.
\end{remark}

In Table \ref{tab1} we summarize  the best known upper bounds on the antipodal
$s$-distance $t$-strength sets so far.

\makeatletter\def\@captype{table}\makeatother
\begin{center}
\caption {Upper bounds on the size of an antipodal $s$-distance set $X\subset\Sd$
with strength $t$ } \label{tab1}
\begin{tabular}{c|cllllllll}
The values of $s$ and $t$    & \multicolumn{9}{c}{An upper bound of $|X|$}
\\ \hline
       $s=3,t=3$    & \multicolumn{9}{c}{$\frac{2d(d+2)}{3}$ (Corollary \ref{wel})}              \\ \hline
$s=4,t=5$    & \multicolumn{9}{c}{$\frac{2d(d+2)^2}{9}$\ (Corollary \ref{coro2})}    \\ \hline

even $s=\frac{t+3}{2}$, $t\geq7$ & \multicolumn{9}{c}{$\min\{2\binom{d+s-2}{s-1}, 2\binom{d+s-1}{s}-2h_{t-s+1}\}$ (\cite{Delsarte} ,\cite{Nozaki})}     \\ \hline

even $s\in[\frac{t+5}{2},t+1]$, $t\geq 3$ &
\multicolumn{9}{c}{$2\binom{d+s-2}{s-1}-2h_{t-s+2}$ (Theorem \ref{general})}     \\
\hline
    odd $s\in[\frac{t+5}{2},t+2]$  & \multicolumn{9}{c}{$2\binom{d+s-2}{s-1}-2h_{t-s+2}$ \cite{Nozaki}}
\end{tabular}
\end{center}
\subsection{Organization}

The paper is organized as follows. In Section \ref{sec2},
 we introduce some definitions and lemmas.
 After presenting the proof of  Theorem \ref{general} in Section \ref{s3},
 we prove Theorem \ref{coro1} in Section \ref{s4}.
We also show the equivalence between the Gerzon bound and the  necessary conditions
on the existence of real ETFs obtained from the Krein conditions
 in Section \ref{s4}. Finally we prove Theorem \ref{lev} in Section
\ref{s5}.


\section{Preliminaries}\label{sec2}

In this section, we introduce some definitions and lemmas which will be used in later
sections.

\subsection{Notations}\label{s21}

Let $\text{Harm}_{k}(\R^d)$ be the  vector space of all real homogeneous harmonic
polynomials of degree $k$ on $d$ variables, equipped with the standard inner product
\begin{equation*}
\langle f,g\rangle=\int_{\Sd}f(\vx)g(\vx)\text{d}\mu_d(\vx)
\end{equation*}
for $f,g\in\text{Harm}_{k}(\R^d)$. It is known that the dimension of $\text{Harm}_{k}(\R^d)$ is $h_k$, where $h_k$ is defined in (\ref{hi}) for each $k\geq 0$ \cite[Theorem 3.2]{Delsarte}.  Let $\{\phi_{k,i}^{(d)}\}_{i=1}^{h_k}$ be an orthonormal basis for $\text{Harm}_{k}(\R^d)$.

Let $G_k^{(d)}(x)$ denote the
Gegenbauer polynomial of degree $k$ with the normalization $G_k^{(d)}(1)=h_k$, which can be
defined recursively as follows (see also \cite[Definition 2.1]{Delsarte}):
\begin{equation*}
G_0^{(d)}(x):=1,\ G_1^{(d)}(x):=d\cdot x,
\end{equation*}
\begin{equation*}
\frac{k+1}{d+2k}\cdot G_{k+1}^{(d)}(x)=x\cdot G_k^{(d)}(x)-\frac{d+k-3}{d+2k-4}\cdot G_{k-1}^{(d)}(x),\ k\geq 1.
\end{equation*}
The following formulation  is  well-known  \cite[Theorem 3.3]{Delsarte}:
\begin{equation}\label{addition}
G_k^{(d)}(\langle \vx,\vy\rangle)=\sum\limits_{i=1}^{h_k}\phi_{k,i}^{(d)}(\vx)\phi_{k,i}^{(d)}(\vy),\
 \text{ for }\ \vx,\vy\in\Sd,\  k\in\mathbb{Z}_+.
\end{equation}

We also need the following notations.

\begin{definition}\label{alldef}
For a finite non-empty set $X\subset\Sd$,  we use the following notations:

\begin{enumerate}[{\rm (i)}]

\item The $k$-th characteristic matrix $\mathbf{H}_k(X)$ of size $|X|\times h_k$ is defined as  (see also \cite[Definition 3.4]{Delsarte}):
\begin{equation*}
\mathbf{H}_k(X):=(\phi_{k,i}^{(d)}(\vx)), \ \vx\in X,\ i\in\{1,2,\ldots,h_k\};
\end{equation*}

\item\label{ii}  Set $\mathbf{D}_k(X):=\mathbf{H}_k(X)\mathbf{H}_k(X)^{\rm T}$
    for each $k\geq 0$;

\item Let $V_k(X)$ denote the direct sum of the eigenspaces corresponding to all positive eigenvalues of $\mathbf{D}_k(X)$ (see also \cite[Page 1707]{Nozaki});

\item The annihilator polynomial of  $X$ is defined as
\begin{equation*}
F_{X}(x):=\prod\limits_{\alpha\in A(X)}\frac{x-\alpha}{1-\alpha};
\end{equation*}

\item\label{vii} When $X$ is antipodal, we say the subset $\hat{X}\subset X$ is a half of $X$ if $\hat{X}$ satisfies $\hat{X}\cap-\hat{X}=\varnothing$ and $ \ \hat{X}\cup-\hat{X}=X$.
\end{enumerate}
\end{definition}

Note that $\mathbf{H}_0(X)$ is exactly the all-ones vector of size $|X|$.  According
to  (\ref{addition}), we have
\begin{equation}\label{D}
\mathbf{D}_k(X)=\mathbf{H}_k(X)\mathbf{H}_k(X)^{\rm T}=(G_k^{(d)}(\langle \vx,\vy\rangle))_{\vx,\vy\in X}.
\end{equation}

Throughout this paper, we use $\mathbf{I}, \mathbf{J}$ to  denote the identity matrix
and all-ones matrix of appropriate size, respectively. We also set
\begin{equation*}
\boldsymbol\Delta_{k,l}:=\left\{
        \begin{array}{cl}
        \mathbf{I}, & {\rm if}\ k=l,\\
          \mathbf{0}, & {\rm otherwise}.
        \end{array}
      \right.
\end{equation*}

\subsection{Spherical designs}\label{s22}

By the notion of characteristic matrices, the following lemma provides two equivalent definitions of spherical $t$-designs.

\begin{lemma}{\rm (see \cite[Theorem 5.3]{Delsarte}) }\label{hk}
A finite set $X\subset\Sd$ is a spherical $t$-design if and only if any one of  the following holds:
\begin{enumerate}[{\rm (i)}]
\item $\mathbf{H}_k(X)^{\rm T}\mathbf{H}_0(X)=\mathbf{0}_{h_k\times 1},\ k=1,2,\ldots,t$.

\item $\mathbf{H}_k(X)^{\rm T}\mathbf{H}_l(X)=|X|\cdot \mathbf \Delta_{k,l}$ when $0\leq k+l\leq t$.

\end{enumerate}
\end{lemma}

We next prove some properties of antipodal spherical designs which will be used in
Section \ref{s3}.

\begin{corollary}\label{coro-anti}
Assume $X\subset\Sd$ is an antipodal set and let $\hat{X}$ be a half of $X$. Then,
\begin{enumerate}[{\rm (i)}]
\item  $X$ is a spherical $t$-design if and only if $\mathbf{H}_k(\hat{X})^{\rm T}\mathbf{H}_0(\hat{X})=\mathbf{0}_{h_k\times 1}$ for each positive even integer $k\leq t$.

\item if $X$ is a spherical $t$-design, then
\begin{equation*}
\mathbf{H}_k(\hat{X})^{\rm T}\mathbf{H}_l(\hat{X})=\frac{|X|}{2}\cdot \mathbf\Delta_{k,l}\quad \text{and}\quad \mathbf{D}_k(\hat{X})\mathbf{D}_l(\hat{X})=\frac{|X|}{2}\cdot \mathbf\Delta_{k,l}\cdot \mathbf{D}_k(\hat{X})
\end{equation*}
hold when $0\leq k+l\leq t$ and $k\equiv l\ {\rm{(mod\ 2)}}$.
\end{enumerate}
\end{corollary}
\begin{proof}
For any $\vx\in\Sd$ and $i\in\{1,2,\ldots,h_k\}$, we have
$\phi_{k,i}^{(d)}(-\vx)=-\phi_{k,i}^{(d)}(\vx)$ if $k$ is odd and
$\phi_{k,i}^{(d)}(-\vx)=\phi_{k,i}^{(d)}(\vx)$ if $k$ is even.  Hence, we have
\begin{equation}\label{eq:Hk}
\mathbf{H}_{k}(X) =
\begin{pmatrix}
\mathbf{H}_k(\hat{X})    \\
(-1)^k\cdot \mathbf{H}_k(\hat{X})  \end{pmatrix}.
\end{equation}

(i) According to (\ref{eq:Hk}) we have
\begin{equation*}
\mathbf{H}_{k}(X)^{\rm T}\mathbf{H}_{0}(X)=\left\{
        \begin{array}{ll}
        2\cdot \mathbf{H}_{k}(\hat{X})^{\rm T}\mathbf{H}_{0}(\hat{X}), & \text{if $k$ is even,}\\
        \mathbf{0}_{h_k\times 1}, & \text{if $k$ is odd.}
        \end{array}
      \right.
\end{equation*}
Based on Lemma \ref{hk} we obtain that $X$ is a spherical $t$-design if and only if
$\mathbf{H}_k(\hat{X})^{\rm T}\mathbf{H}_0(\hat{X})=\mathbf{0}_{h_k\times 1}$ for
each positive even integer $k\leq t$.

(ii) Let $k$ and $l$ be two integers satisfying $0\leq k+l\leq t$  and $k\equiv l\
{\rm{(mod\ 2)}}$. Equation (\ref{eq:Hk}) implies  $\mathbf{H}_{k}(X)^{\rm
T}\mathbf{H}_{l}(X)=2\cdot \mathbf{H}_{k}(\hat{X})^{\rm T}\mathbf{H}_{l}(\hat{X})$.
Thus, according to  Lemma \ref{hk} we obtain that $\mathbf{H}_k(\hat{X})^{\rm
T}\mathbf{H}_l(\hat{X})=\frac{|X|}{2}\cdot \mathbf\Delta_{k,l}$ if $X$ is a spherical
$t$-design. According to (\ref{ii}) in Definition \ref{alldef}, we have
\[
\mathbf{D}_k(\hat{X})\mathbf{D}_l(\hat{X})=\mathbf{H}_k(\hat{X})\mathbf{H}_k(\hat{X})^T
\mathbf{H}_l(\hat{X})\mathbf{H}_l(\hat{X})^T= \frac{|X|}{2}\cdot
\mathbf\Delta_{k,l}\cdot \mathbf{D}_k(\hat{X}).
\]
\end{proof}

The following lemma played a key role in Nozaki and Suda's framework  \cite{Nozaki}. Its
main idea is to identify the size of an $s$-distance set $X$ with the dimension of a sum of subspaces
$V_k(X)$ defined in Definition \ref{alldef}.

\begin{lemma}{\rm (see \cite[Lemma 3.2]{Nozaki}) }\label{dim}
Let $X\subset\Sd$ be an $s$-distance set. Assume the annihilator polynomial of $X$ has the Gegenbauer expansion $F_X(x)=\sum\limits_{k=0}^{s}f_kG_k^{(d)}(x)$. Then we have $|X|={\rm{dim}}(\sum\limits_{k:f_k>0}V_k(X))$.
\end{lemma}


\subsection{Spherical embeddings of strongly regular graphs}\label{s24}

In this subsection we briefly introduce the spherical embeddings of  strongly regular
graphs, which will be used in our analysis of Levenstein-equality packings in Section
\ref{s5}. A regular graph $\Gamma$ with $v$ vertices and degree $k$ is called
strongly regular if every two adjacent vertices have $\lambda$ common neighbors and
every two non-adjacent vertices have $\mu$ common neighbors.
Let $\Gamma$ be a strongly regular graph with parameters $(v,k,\lambda,\mu)$.
 Denote
its vertex set by $\{1,2,\ldots,v\}$ for simplicity. The adjacency matrix
$\mathbf{A}$ of $\Gamma$ has three eigenvalues $k$, $r_1$ and $r_2$, with
multiplicities $1$, $n_1$ and $n_2$, respectively. The values of $r_1,r_2,n_1,n_2$
can be calculated as follows \cite{Bondarenko, Cameron2}
\begin{equation}\label{r12}
r_1=\frac{1}{2}(\lambda-\mu+\sqrt{(\lambda-\mu)^2+4(k-\mu)}), \quad r_2=\frac{1}{2}(\lambda-\mu-\sqrt{(\lambda-\mu)^2+4(k-\mu)}),
\end{equation}
\begin{equation}\label{n12}
n_1=\frac{1}{2}(v-1-\frac{2k+(v-1)(\lambda-\mu)}{\sqrt{(\lambda-\mu)^2+4(k-\mu)}}), \quad n_2=\frac{1}{2}(v-1+\frac{2k+(v-1)(\lambda-\mu)}{\sqrt{(\lambda-\mu)^2+4(k-\mu)}}).
\end{equation}

For each $i\in\{1,2\}$, let $E_i$ denote the eigenspace of $\mathbf{A}$ with respect to the eigenvalue $r_i$. Then a  spherical embedding of $\Gamma$ with respect to $E_i$ is a collection of unit vectors in $\R^{n_i}$, obtained by orthogonally projecting a standard basis of $\R^v$ onto the eigenspace $E_i$ and rescaling the projections to have unit norm. It is known that the obtained set is a two-distance spherical $2$-design \cite{Bondarenko, Cameron2}. If we let $Y^{(i)}=\{\vy_j^{(i)}\}_{j=1}^{v}$ denote the spherical embedding of $\Gamma$ with respect to $E_i$, $i\in\{1,2\}$, then we have \cite{Bondarenko,Cameron2,Bannai4} :
\begin{equation}\label{embed}
\absinn{\vy_j^{(i)},\vy_l^{(i)}}=\left\{
        \begin{array}{cl}
        1, & {\rm{if }}\ j=l,\\
         \frac{r_i}{k},  & {\text{if vertex $j$ and vertex $l$ are adjacency}},\\
          -\frac{r_i+1}{v-k-1}, & \text{otherwise}.
        \end{array}
      \right.
\end{equation}

In Section \ref{s5} we will introduce that each Levenstein-equality packing gives
rise to a strongly regular graph. Then we will use one of the spherical embeddings of
this strongly regular graph to provide a lower bound on the size of
Levenstein-equality packings.


\section{Proof of Theorem \ref{general}}\label{s3}

In this section, motivated by the method developed in \cite{Nozaki}, we  present a
proof of Theorem \ref{general}.

Assume $s$ is an even integer. Let $X\subset\Sd$ be an antipodal $s$-distance set
with strength $t$ and let $\hat{X}$ be a half of $X$ (see (\ref{vii}) in Definition  \ref{alldef}). Now we focus on estimating the maximum size of $\hat{X}$.  Noting that
$X$ is antipodal and $s$ is even, we assume
$A(X)=\{-1,0,\pm\alpha_1,\ldots,\pm\alpha_{\frac{s-2}{2}}\}$, where
$\alpha_i\in(0,1)$ for each $i\in\{1,2,\ldots,\frac{s-2}{2}\}$. Noting that
$A(\hat{X})=A(X)\backslash\{-1\}$,  we have
\begin{equation*}
F_{\hat{X}}(x)=x\cdot\prod\limits_{i=1}^{\frac{s-2}{2}}\frac{x^2-\alpha_i^2}{1-\alpha_i^2}.
\end{equation*}
It follows that $F_{\hat{X}}(x)$ is an odd function. Assume that $F_{\hat{X}}(x)$
has the Gegenbauer expansion $F_{\hat{X}}(x)=\sum\limits_{k=0}^{s-1}f_kG_k^{(d)}(x)$.
It is well known that the Gegenbauer polynomial $G_k^{(d)}(x)$ is an odd function if
$k$ is  odd and an even function if $k$ is even \cite[Page 59]{Szego}. This means
that $f_k=0$ provided $k\leq s-1$ is even. Hence, by Lemma \ref{dim} we obtain
\begin{equation}\label{antipodal}
|\hat{X}|={\rm{dim}}(\sum\limits_{k:f_k>0}V_k(\hat{X}))\leq\text{dim}(\sum\limits_{k=0}^{\frac{s-2}{2}}\ V_{2k+1}(\hat{X})).
\end{equation}
Now we aim to prove that $V_{t-s+2}(\hat{X})$ is contained in the sum of some other
subspaces $V_{2k+1}(\hat{X})$ when $s\in [\frac{t+5}{2}, t+1]$. The following lemma
is analogous to \cite[Lemma 3.3]{Nozaki}.

\begin{lemma}\label{lemma1}
Suppose $X\subset\Sd$ is an antipodal $s$-distance set with strength $t$,  where
$s\in [\frac{t+5}{2}, t+1]$ is an even integer and $t\geq 3$ is an odd integer. Let
$\hat{X}$ be a half of $X$.  Assume the annihilator polynomial of $\hat{X}$ has the
Gegenbauer expansion
$F_{\hat{X}}(x)=\sum\limits_{k=0}^{\frac{s-2}{2}}f_{2k+1}G_{2k+1}^{(d)}(x)$.  If
$f_{2i+1}\neq\frac{1}{|\hat{X}|}$ for some integer $i$ satisfying $t-s+2\leq 2i+1\leq
\frac{t-1}{2}$, then we have
$V_{2i+1}(\hat{X})\subset\sum\limits_{k=\frac{t-2i-1}{2}}^{\frac{s-2}{2}}V_{2k+1}(\hat{X})$.
\end{lemma}
\begin{proof}
Set $\mathbf{F}:=(F_{\hat{X}}(\left<{\vx,\vy}\right>))_{\vx,\vy\in \hat{X}}$. Noting
that $F_{\hat{X}}(1)=1$ and $F_{\hat{X}}(\alpha)=0$ for $\alpha\in A(\hat{X})$, we
obtain that $\mathbf{F}$ is exactly the identity matrix of size $|\hat{X}|$. On the
other hand, by the Gegenbauer expansion of $F_{\hat{X}}(x)$ and (\ref{D}), we have
\begin{equation}\label{F1}
\mathbf{I}=\mathbf{F}=\sum\limits_{k=0}^{\frac{s-2}{2}}f_{2k+1}\cdot \mathbf{D}_{2k+1}(\hat{X}).
\end{equation}

For each integer $i$ satisfying $t-s+2\leq 2i+1\leq \frac{t-1}{2}$, we multiply $\mathbf{D}_{2i+1}(\hat{X})$ on both sides of (\ref{F1}) and obtain
\begin{equation}\label{F2}
\begin{aligned}
\mathbf{D}_{2i+1}(\hat{X})&=\sum\limits_{k=0}^{\frac{s-2}{2}}f_{2k+1}\cdot \mathbf{D}_{2k+1}(\hat{X})\mathbf{D}_{2i+1}(\hat{X})\\
&=f_{2i+1}\cdot \mathbf{D}_{2i+1}(\hat{X})\mathbf{D}_{2i+1}(\hat{X})+\sum\limits_{k=0,\ k\neq i}^{\frac{s-2}{2}}f_{2k+1}\cdot \mathbf{D}_{2k+1}(\hat{X})\mathbf{D}_{2i+1}(\hat{X}).
\end{aligned}
\end{equation}
Since $X$ is an antipodal spherical $t$-design, by Corollary \ref{coro-anti}, we have
\begin{equation}\label{F3}
\mathbf{D}_{2i+1}(\hat{X})=|\hat{X}|\cdot f_{2i+1}\cdot \mathbf{D}_{2i+1}(\hat{X})+
\sum_{k=\frac{t-2i-1}{2}}^{\frac{s-2}{2}}f_{2k+1}\cdot \mathbf{D}_{2k+1}(\hat{X})\mathbf{D}_{2i+1}(\hat{X}).
\end{equation}
Noting that $s\in [\frac{t+5}{2}, t+1]$ is an even integer and $t\geq 3$ is an odd
integer, rearranging equation (\ref{F3}) gives
\begin{equation}\label{F4}
(1-|\hat{X}|\cdot f_{2i+1})\mathbf{D}_{2i+1}(\hat{X})=\sum\limits_{k=\frac{t-2i-1}{2}}^{\frac{s-2}{2}}f_{2k+1}\cdot \mathbf{D}_{2k+1}(\hat{X})\mathbf{D}_{2i+1}(\hat{X}).
\end{equation}

Assume $\mathbf{v}$ is an eigenvector of $\mathbf{D}_{2i+1}(\hat{X})$ with respect to
an eigenvalue $\lambda\neq 0$. Then, we have
\begin{equation*}
(1-|\hat{X}|\cdot f_{2i+1})\mathbf{D}_{2i+1}(\hat{X})\mathbf{v}=\sum\limits_{k=\frac{t-2i-1}{2}}^{\frac{s-2}{2}}f_{2k+1}\cdot \mathbf{D}_{2k+1}(\hat{X})\mathbf{D}_{2i+1}(\hat{X})\mathbf{v},
\end{equation*}
which implies
\begin{equation}\label{eq:L2}
(1-|\hat{X}|\cdot
f_{2i+1})\cdot\lambda\mathbf{v}=\lambda\cdot\sum\limits_{k=\frac{t-2i-1}{2}}^{\frac{s-2}{2}}f_{2k+1}\cdot
\mathbf{D}_{2k+1}(\hat{X})\mathbf{v}.
\end{equation}
 Note that a real symmetric matrix of size $|\hat{X}|$
always has $|\hat{X}|$ linear independent eigenvectors. Hence, we can write
$\mathbf{v}=\sum_{j=1}^{|\hat{X}|}\mathbf{v}_j^{(2k+1)}$  for each
$k\in[\frac{t-2i-1}{2}, \frac{s-2}{2}]$, where
$\{\mathbf{v}_j^{(2k+1)}\}_{j=1}^{|\hat{X}|}$ is a set of linear independent
eigenvectors of $\mathbf{D}_{2k+1}(\hat{X})$. Assume that $\lambda_j^{(2k+1)}$ is an
eigenvalue of $\mathbf{D}_{2k+1}(\hat{X})$ with respect to the eigenvector
$\mathbf{v}_j^{(2k+1)}$. Since $\mathbf{D}_{2k+1}(\hat{X})$ is a symmetric positive
semidefinite matrix,  we have $\lambda_j^{(2k+1)}\geq0$ for each
$j\in\{1,2,\ldots,|\hat{X}|\}$. Then, according to (\ref{eq:L2}), we have
\begin{equation*}
\begin{aligned}
(1-|\hat{X}|\cdot f_{2i+1})\cdot\lambda\mathbf{v}=\lambda\cdot\sum\limits_{k=\frac{t-2i-1}{2}}^{\frac{s-2}{2}}f_{2k+1}\cdot \mathbf{D}_{2k+1}(\hat{X})\sum_{j=1}^{|\hat{X}|}\mathbf{v}_j^{(2k+1)},
\end{aligned}
\end{equation*}
which implies
\[
(1-|\hat{X}|\cdot
f_{2i+1})\cdot\lambda\mathbf{v}=\lambda\cdot\sum\limits_{k=\frac{t-2i-1}{2}}^{\frac{s-2}{2}}f_{2k+1}\cdot
\sum_{{j:\ \lambda_j^{(2k+1)}>0} }\lambda_j^{(2k+1)}\mathbf{v}_j^{(2k+1)}.
\]
Thus, if  $f_{2i+1}\neq\frac{1}{|\hat{X}|}$, then  $\mathbf{v}$ can be written  as a
linear combination of vectors in
$\sum\limits_{k=\frac{t-2i-1}{2}}^{\frac{s-2}{2}}V_{2k+1}(\hat{X})$, which implies
that
$\mathbf{v}\in\sum\limits_{k=\frac{t-2i-1}{2}}^{\frac{s-2}{2}}V_{2k+1}(\hat{X})$.
Since $\mathbf{v}$ can be any vector in $V_{2i+1}(\hat{X})$, we arrive at our
conclusion.
\end{proof}

It remains to show that the coefficient $f_{t-s+2}$ in the Gegenbauer expansion of $F_{\hat{X}}(x)$ is not $\frac{1}{|\hat{X}|}$. We need the following lemma.

\begin{lemma}{\rm \cite[Lemma 2.6]{Delsarte} }\label{lemma2}
Assume $F(x)=\sum_kf_kG_k^{(d)}(x)$ and let $Q(x):=\frac{G_l^{(d)}(x)}{h_l}\cdot F(x)$ for some positive integer $l$. Assume $Q(x)$ has the Gegenbauer expansion $Q(x)=\sum_kq_k\cdot G_k^{(d)}(x)$. Then $q_0=f_l$.
\end{lemma}

With the help of the above lemma, we now show that the coefficient $f_{t-s+2}$ in the Gegenbauer expansion of $F_{\hat{X}}(x)$ is not $\frac{1}{|\hat{X}|}$. Actually, we prove that $f_{t-s+2}$ is the first coefficient with this property.

\begin{lemma}\label{lemma3}
Suppose $X\subset\Sd$ is an antipodal $s$-distance set with strength $t$, where $s\in[\frac{t+3}{2},t+1]$ is an even integer and $t\geq 3$ is an odd integer. Let $\hat{X}$ be a half of $X$.  Assume the annihilator polynomial of $\hat{X}$ has the Gegenbauer expansion $F_{\hat{X}}(x)=\sum\limits_{k=0}^{\frac{s-2}{2}}f_{2k+1}G_{2k+1}^{(d)}(x)$. Then, $f_{t-s+2}\neq \frac{1}{|\hat{X}|}$ and $f_{l-s+2}=\frac{1}{|\hat{X}|}$ for each odd integer $l$ satisfying $s-1\leq l<t$.
\end{lemma}

\begin{proof}
Set  $Q_l(x):=\frac{G_{l-s+2}^{(d)}(x)}{h_{l-s+2}}\cdot F_{\hat{X}}(x)$ for each odd
integer $l$ satisfying $s-1\leq l\leq t$. Since $F_{\hat{X}}(x)$ is a polynomial of
degree $s-1$, we see that $Q_l(x)$ is a polynomial of degree $l+1$. Noting that both
$G_{l-s+2}^{(d)}(x)$ and $F_{\hat{X}}(x)$ are odd functions,  we obtain that $Q_l(x)$
is an even function. Thus we can assume $Q_l(x)$ has the Gegenbauer expansion
$Q_l(x)=\sum\limits_{i=0}^{\frac{l+1}{2}}q_{2i}^{(l)}\cdot G_{2i}^{(d)}(x)$. Since
$F_{\hat{X}}(1)=1$ and $F_{\hat{X}}(\alpha)=0$ for each $\alpha\in A(\hat{X})$, we
obtain that $Q_l(1)=1$ and $Q_l(\alpha)=0$ for each $\alpha\in A(\hat{X})$. This
implies
\begin{equation}\label{sumQ}
\sum\limits_{\vx,\vy\in \hat{X}}Q_l(\langle \vx,\vy\rangle)=|\hat{X}|.
\end{equation}
On the other hand, we have
\begin{equation}\label{Q0}
\sum\limits_{\vx,\vy\in \hat{X}}Q_l(\langle \vx,\vy\rangle)=\sum\limits_{\vx,\vy\in \hat{X}}\sum\limits_{i=0}^{\frac{l+1}{2}}q_{2i}^{(l)}\cdot G_{2i}^{(d)}(\langle x,y\rangle)=|\hat{X}|^2q_0^{(l)}+\sum\limits_{i=1}^{\frac{l+1}{2}}(q_{2i}^{(l)}\cdot\sum\limits_{\vx,\vy\in \hat{X}} G_{2i}^{(d)}(\langle \vx,\vy\rangle))
\end{equation}
Note that $\mathbf{H}_{0}(\hat{X})$ is the all-ones vector of size $|\hat{X}|$.
According to  (\ref{D}), for each $i\in\{1,2,\ldots,\frac{l+1}{2}\}$, we have
\begin{equation}\label{Q00}
\begin{aligned}
\sum\limits_{\vx,\vy\in \hat{X}}  G_{2i}^{(d)}(\langle \vx,\vy\rangle)&=\mathbf{H}_{0}(\hat{X})^T \mathbf{D}_{2i}(\hat{X})\mathbf{H}_{0}(\hat{X})\\
&=\mathbf{H}_{0}(\hat{X})^T\mathbf{H}_{2i}(\hat{X})\mathbf{H}_{2i}(\hat{X})^T\mathbf{H}_{0}(\hat{X})=||\mathbf{H}_{2i}(\hat{X})^{\rm T}\mathbf{H}_{0}(\hat{X})||_2^2.
\end{aligned}
\end{equation}
Combining  (\ref{Q0}) and (\ref{Q00}), we obtain
\begin{equation}\label{Q01}
\sum\limits_{\vx,\vy\in \hat{X}}Q_l(\langle \vx,\vy\rangle)=|\hat{X}|^2q_0^{(l)}+\sum\limits_{i=1}^{\frac{l+1}{2}}q_{2i}^{(l)}\cdot ||\mathbf{H}_{2i}(\hat{X})^{\rm T}\mathbf{H}_{0}(\hat{X})||_2^2.
\end{equation}
Combining  (\ref{sumQ}) and (\ref{Q01}), we arrive at
\begin{equation}\label{Q1}
|\hat{X}|-|\hat{X}|^2q_0^{(l)}=\sum\limits_{i=1}^{\frac{l+1}{2}}q_{2i}^{(l)}\cdot ||\mathbf{H}_{2i}(\hat{X})^{\rm T}\mathbf{H}_{0}(\hat{X})||_2^2.
\end{equation}
{ Since $X$ has strength $t$, by Corollary \ref{coro-anti} we have
\begin{equation}\label{H2i}
||\mathbf{H}_{2i}(\hat{X})^{\rm T}\mathbf{H}_{0}(\hat{X})||_2^2=0,\ \text{ for all }\  i\in\{1,2,\ldots,\frac{t-1}{2}\}
\end{equation}
and
\begin{equation}\label{Ht1}
||\mathbf{H}_{t+1}(\hat{X})^{\rm T}\mathbf{H}_{0}(\hat{X})||_2^2\neq 0.
\end{equation}
Then, by equation (\ref{Q1}) we obtain
\begin{equation}\label{ql}
|\hat{X}|-|\hat{X}|^2q_0^{(l)}=0
\end{equation}
when $s-1\leq l< t$ and
\begin{equation}\label{qt}
|\hat{X}|-|\hat{X}|^2q_0^{(t)}=q_{t+1}^{(t)}\cdot
\|\mathbf{H}_{t+1}(\hat{X})^{\rm{T}}\mathbf{H}_{0}(\hat{X})\|_2^2.
\end{equation}
The (\ref{ql}) implies  $q_0^{(l)}=\frac{1}{|\hat{X}|}$ for each odd integer $s-1\leq
l<t$. Noting that $Q_t(x)$ is a polynomial of degree $t+1$, we have
$q_{t+1}^{(t)}\neq 0$. Combining (\ref{Ht1}) and (\ref{qt}), we obtain
$q_0^{(t)}\neq\frac {1}{|\hat{X}|}$. } By Lemma \ref{lemma2} we know that
$q_0^{(l)}=f_{l-s+2}$ for each $l\geq s-1$. Hence, we arrive at our conclusion.
\end{proof}

 We next present a proof of Theorem \ref{general}.

\begin{proof}[Proof of Theorem \ref{general}]

Recall that $X$ is an antipodal $s$-distance set with strength $t$. Let $\hat{X}$ be a half of $X$. Combining Lemma \ref{lemma1} and Lemma \ref{lemma3}, we know that $V_{t-s+2}(\hat{X})$ is contained in $V_{s-1}(\hat{X})$. Then by (\ref{antipodal}) we have
\begin{equation}\label{nonantipodal}
\begin{aligned}
|\hat{X}|&\leq\text{dim}(\sum\limits_{k=0}^{\frac{s-2}{2}}\ V_{2k+1}(\hat{X}))
=\text{dim}(\sum\limits_{\substack{k=0 \\ k\neq \frac{t-s+1}{2}}}^{\frac{s-2}{2}}\ V_{2k+1}(\hat{X}))\\
&\leq \sum\limits_{\substack{k=0 \\ k\neq \frac{t-s+1}{2}}}^{\frac{s-2}{2}}\text{dim}\ V_{2k+1}(\hat{X})\leq \sum\limits_{k=0}^{\frac{s-2}{2}} h_{2k+1}-h_{t-s+2}=\binom{d+s-2}{s-1}-h_{t-s+2}.
\end{aligned}
\end{equation}
The last inequality in (\ref{nonantipodal}) follows from
\begin{equation*}
\text{dim}\ V_{2k+1}(X)=\text{rank}\ (\mathbf{D}_{2k+1}(X))\leq \text{rank}\ (\mathbf{H}_{2k+1}(X))\leq
h_{2k+1}.
\end{equation*}
Noting that $|X|=2|\hat{X}|$, we obtain $|X|\leq 2\binom{d+s-2}{s-1}-2 h_{t-s+2}$.
\end{proof}

\begin{remark}
Lemma \ref{lemma1} and Lemma \ref{lemma3} can be easily extended to the case when $s$ is an odd integer. Using these extended results one can obtain an upper bound on $|X|$ for odd $s\in[\frac{t+5}{2},t+2]$, which is actually the same with the bound in (\ref{odds}). Hence, for clarity and convenience we only consider the case when $s$ is an even integer in Lemma \ref{lemma1} and Lemma \ref{lemma3} .
\end{remark}


\section{Proof of Theorem \ref{coro1}}\label{s4}

The aim of this section is to present a proof of Theorem \ref{coro1}.
We need the following necessary condition on the existence of real ETFs.

\begin{lemma}{\rm (Theorem A in \cite{ETF2})}\label{aw}
Let $d$ and $n$ be two integers satisfying $n>d+1>2$ and $n\neq 2d$. If there exists an ETF for $\R^d$ with size $n$, then both $\sqrt{\frac{d(n-1)}{n-d}}$ and $\sqrt{\frac{(n-d)(n-1)}{d}}$ are odd integers.
\end{lemma}

Using the above lemma, we present a proof of Theorem \ref{coro1}.

\begin{proof}[Proof of Theorem \ref{coro1}]

Recall that $\Phi$ is an ETF for $\R^d$ with size $n>d+1\geq6$. We first show that
$n\leq\frac{d(d+2)}{3}$ if $n\neq\frac{d(d+1)}{2}$. Notice that $2d<\frac{d(d+2)}{3}$
when $d\geq 5$, so we only need to consider the case when $n>2d$.   According to
Lemma \ref{aw}, we can assume $\sqrt{\frac{d(n-1)}{n-d}}=2k-1$ for some positive
integer $k>1$. Then a simple calculation shows:
\begin{subequations}
\begin{align}
(2k-1)^2(n-d)&=d(n-1)\\
((2k-1)^2-d)(n-d)&=d(d-1).\label{etf1}
\end{align}
\end{subequations}
Since $n-d$ and $d(d-1)$ are positive, from (\ref{etf1}) we obtain  that $(2k-1)^2-d$
is a positive integer. If $(2k-1)^2-d=1$, then (\ref{etf1}) gives $n=d^2$. This is
impossible since $n$ must satisfy the Gerzon bound (\ref{Gerzon}). Hence, we must
have $(2k-1)^2-d\geq2$. If $(2k-1)^2-d=2$, then (\ref{etf1}) gives
$n=\frac{d(d+1)}{2}$; otherwise, we have $(2k-1)^2-d\geq3$, then (\ref{etf1}) implies
$n\leq\frac{d(d+2)}{3}$. Hence, we have $n\leq\frac{d(d+2)}{3}$ if
$n\neq\frac{d(d+1)}{2}$.

It remains to prove that $n\geq d+\frac12+\sqrt{3d+\frac14}$ if $n\neq d+\frac12+\sqrt{2d+\frac14}$. Notice that $d+\frac12+\sqrt{3d+\frac14}<2d$ when $d\geq5$, so we only need to consider the case when $n<2d$. Set $m:=n-d$. By Lemma \ref{aw}, we may assume $\sqrt{\frac{m(n-1)}{n-m}}=\sqrt{\frac{(n-d)(n-1)}{d}}=2p-1$ for some positive integer $p$. By similar computation with (\ref{etf1}) we obtain
\begin{equation}\label{etf2}
((2p-1)^2-m)(n-m)=m(m-1).
\end{equation}
Since $n-m=d>0$ and $m=n-d>1$, from (\ref{etf2}) we see that $(2p-1)^2-m$ is a
positive  integer. If $(2p-1)^2-m=1$, then (\ref{etf2}) gives $n=m^2$, that is,
$n=d+\frac12+\sqrt{d+\frac14}$. This is impossible since $n$ must satisfy the Gerzon
bound (\ref{Gerzon}). Hence, we have $(2p-1)^2-m\geq2$. If $(2p-1)^2-m=2$, then we
have $n=\frac{m(m+1)}{2}$, that is, $n=d+\frac12+\sqrt{2d+\frac14}$; otherwise, we
have $(2p-1)^2-m\geq3$, then (\ref{etf2}) implies  $n\leq\frac{m(m+2)}{3}$, that is,
$n\geq d+\frac12+\sqrt{3d+\frac14}$. Hence, we must have $n\geq
d+\frac12+\sqrt{3d+\frac14}$ if $n\neq d+\frac12+\sqrt{2d+\frac14}$. Putting all
these together, we arrive at the conlcusion.
\end{proof}

For the remainder of this section we  compare the Krein conditions for strongly
regular graph with Gerzon bound. It is well known that there exists an ETF for $\R^d$
with size $n>d+1>2$ if and only if there exists a strongly regular graph with
parameters $(n-1,a,\frac{3a-n}{2},\frac{a}{2})$ \cite{ETF2,Waldron,ETF}, where $a$ is
defined in (\ref{a}). It is also known that each strongly regular graph satisfies the
following  Krein conditions:

\begin{lemma}{\rm{(see \cite{Scott,Brouwer2,ETF} )}}\label{krein}
Assume there exists a strongly regular graph $\Gamma$ with given parameters
$v,k,\lambda,\mu$.  Then the parameters $v,k,\lambda,\mu$ satisfy the following Krein
conditions:
\begin{subequations}\label{krein2}
\begin{align}
K_1:=(k+r_1)(r_2+1)^2-(r_1+1)(k+r_1+2r_1r_2)\geq0,\label{k1}\\
K_2:=(k+r_2)(r_1+1)^2-(r_2+1)(k+r_2+2r_1r_2)\geq0,\label{k2}
\end{align}
\end{subequations}
where  $r_1$ and $r_2$ are defined  in (\ref{r12}).
\end{lemma}

Hence, if we apply the above lemma to strongly regular  graphs with parameters
$(n-1,a,\frac{3a-n}{2},\frac{a}{2})$, then (\ref{k1}) and (\ref{k2}) provide two
necessary conditions on the existence of nontrivial ETFs. The
authors of \cite{Waldron} and \cite{ETF} wondered whether these two necessary
conditions are covered by the Gerzon bound (\ref{Gerzon}) or other known necessary
conditions. In what follows we show that they are actually equivalent to the Gerzon
bound (\ref{Gerzon}).
\begin{prop}\label{pr:krein}
Assume that $n>d+1>2$. Set
\begin{equation}\label{eq:a}
a:=\frac{n}{2}-1+(1-\frac{n}{2d})\sqrt{\frac{d(n-1)}{n-d}}.
\end{equation}
The $(n,d)$ satisfies Krein conditions (\ref{k1}) and (\ref{k2})
 with parameters $v=n-1, k=a, \lambda=\frac{3a-n}{2},\mu=\frac{a}{2}$
 if and only if $(n,d)$ satisfies the  Gerzon
bound (\ref{Gerzon}).
\end{prop}
\begin{proof}
Substituting $v=n-1, k=a, \lambda=\frac{3a-n}{2},\mu=\frac{a}{2}$ into equation (\ref{r12}), we can represent $r_1$ and $r_2$ as follows:
\begin{subequations}\label{rs2}
\begin{align}
r_1&=\frac{1}{2}\cdot\sqrt{\frac{d(n-1)}{n-d}}-\frac{1}{2},\label{r2}\\
r_2&=-\frac{1}{2}\cdot\frac{n-d}{d}\cdot\sqrt{\frac{d(n-1)}{n-d}}-\frac{1}{2}.\label{s2}
\end{align}
\end{subequations}
Next, substituting equation  (\ref{eq:a}), (\ref{r2}), (\ref{s2}) into (\ref{k1}) and (\ref{k2}), we obtain
\begin{subequations}\label{krein3}
\begin{align}
K_1&=\frac{n}{8d}\cdot\sqrt{\frac{n-1}{n-d}}\cdot (\sqrt{\frac{n-1}{n-d}}-\sqrt{\frac{1}{d}})\cdot(n^2-(2d+1)n+d^2-d),\label{k11}\\
K_2&=\frac{n}{8d(n-d)}\cdot(n-1+\sqrt{\frac{d(n-1)}{n-d}})\cdot(d^2+d-2n).\label{k22}
\end{align}
\end{subequations}
Since $n>d+1>2$,  $K_1\geq 0$ if and only if $n^2-(2d+1)n+d^2-d\geq0$, i.e.,  $n\geq
d+\frac12+\sqrt{2d+\frac14}$. Moreover, $K_2\geq 0$ if and only if $n\leq
\frac{d(d+1)}{2}$. We arrive at the conclusion.
\end{proof}


\section{Proof of Theorem  \ref{lev}}\label{s5}

The main goal of this section is to prove Theorem \ref{lev}.
For convenience, in the rest of this paper, we use $\alpha_{n,d}$ to denote the
Levenstein bound in (\ref{leven}), i.e.,
\begin{equation}\label{aleven}
\alpha_{n,d}\,\,:=\,\,\sqrt{\frac{3n-d(d+2)}{(d+2)(n-d)}}.
\end{equation}
We assume that $\Phi\subset\Sd$ with $|\Phi|=n$ is a Levenstein-equality packing,
i.e., $\mu(\Phi)=\alpha_{n,d}$. Hence $\Phi$ has the angle set
$\{0,\alpha_{n,d},-\alpha_{n,d}\}$ {(see \cite[Example 8.4]{Delsarte} and
\cite[Proposition 3.3]{Haas})}.

We begin with introducing two basic properties about Levenstein-equality packings.
The following lemma says that each Levenstein-equality packing gives rise to a
strongly regular graph.

\begin{lemma}{\rm{\cite[Page 83]{Neumaier}}}\label{depend}
Assume $\Phi=\{\boldsymbol\varphi_i\}_{i=1}^{n}\subset\Sd$ is a Levenstein-equality
packing with the angle set $\{0,\alpha_{n,d},-\alpha_{n,d}\}$, where $\alpha_{n,d}$
is defined in (\ref{aleven}). Let $\Gamma$ be a graph with $n$ vertices where vertex
$i$ and vertex $j$ are adjacency if $\langle
\boldsymbol\varphi_i,\boldsymbol\varphi_j\rangle\neq 0$. Then $\Gamma$ is a strongly
regular graph with parameters $(n,k,\lambda,\mu)$, where
\begin{equation}\label{para1}
k=-\frac{r_2(3\cdot r_1-r_2)}{2},\lambda=r_1+r_2+\mu, \mu=-\frac{r_2(r_1-r_2)}{2},r_1=\frac{2n}{3d}+\frac{r_2}{3}-\frac{2}{3},r_2=-\frac{1}{\alpha_{n,d}^2}.
\end{equation}
Here, $r_1$ and $r_2$ are defined in (\ref{r12}).
\end{lemma}

\begin{remark}
Substituting (\ref{aleven}) into (\ref{para1}) we can write these parameters in terms
of $n$ and $d$ as follows:
\begin{subequations}\label{para2}
\begin{align}
k&=\frac{(n-d)^2(d+2)}{d\cdot(3n-d(d+2))},\label{kk}\\
\lambda&=\frac{(n-d)\cdot((d+8)n^2-9d(d+2)n+2d^2(d+2)^2)}{d\cdot(3n-d(d+2))^2},\label{ll}\\
\mu&=\frac{(n-d)^2(d+2)n}{d\cdot(3n-d(d+2))^2},\label{mm}\\
r_1&=\frac{(n-d)(2n-d(d+2))}{d\cdot(3n-d(d+2))},\ r_2=-\frac{(n-d)(d+2)}{3n-d(d+2)}.\label{rr}
\end{align}
\end{subequations}
Recall that $r_1$ and $r_2$ are the eigenvalues of $\mathbf{A}$, which is the
adjacency matrix of $\Gamma$, with multiplicities $n_1$ and $n_2$. Substituting
(\ref{kk}), (\ref{ll}) and (\ref{mm}) into (\ref{n12}), we can obtain the
multiplicities of $r_1$ and $r_2$ as follows
\begin{equation}\label{nn}
n_1=\frac{d(d+1)}{2}-1,\ n_2=n-\frac{d(d+1)}{2}.
\end{equation}
\end{remark}

The next lemma introduces another property of Levenstein-equality packings.

\begin{lemma}\label{al}
Let $d\geq4$ be an integer. If $\Phi=\{\boldsymbol\varphi_i\}_{i=1}^n\subset\Sd$ is a
Levenstein-equality packing with size $n$, then both
$\frac{1}{\alpha_{n,d}}$ and $\frac{n-d}{d\cdot \alpha_{n,d}}$ are integers.
\end{lemma}

\begin{proof}
Let $\mathbf{M}$ be the matrix of size $d\times n$ whose $i$-th column is
$\boldsymbol\varphi_i$. Since $\Phi$ is a Levenstein-equality packing,
$\Phi\cup-\Phi$ is an antipodal spherical $5$-design. Note that every spherical
2-design is a unit norm tight frame \cite[Proposition 6.1]{Waldron2}. Hence, $\Phi$
forms a unit norm tight frame in $\R^d$, i.e., $\mathbf{M}\mathbf{M}^{\rm
T}=\frac{n}{d}\cdot\mathbf{I}_{d\times d}$ which has eigenvalue $n/d$ with
multiplicities $d$. Set $\mathbf{G}:=\frac{1}{\alpha_{n,d}}\cdot(\mathbf{M}^{\rm
T}\mathbf{M}-\mathbf{I}_{n\times n})$. Since the nonzero eigenvalues of
$\mathbf{M}^{\rm T}\mathbf{M}$ and $\mathbf{M}\mathbf{M}^{\rm T}$ have the same value
and the same algebraic multiplicity, we see that $\mathbf{G}$ has two different
eigenvalues:
\begin{equation}\label{eigen}
\lambda_1=-\frac{1}{\alpha_{n,d}}\quad \text{and}\quad \lambda_2=\frac{1}{\alpha_{n,d}}\cdot\frac{n-d}{d}
\end{equation}
with multiplicities $n-d$ and $d$, respectively.  Moreover, since the $(i,j)$-entry
of $\mathbf{M}^{\rm T}\mathbf{M}$ is the inner product between $\boldsymbol\varphi_i$
and $\boldsymbol\varphi_j$, $\mathbf{G}$ is a matrix whose diagonal entries are all
zeros and non-diagonal entries are $0$ or $\pm 1$. This means that both $\lambda_1$
and $\lambda_2$ are algebraic integers. Since an algebraic integer is an integer if
it is a rational number, it remains to prove that both $\lambda_1$ and $\lambda_2$
are rational numbers. For the aim of contradiction, we assume  that $\lambda_1$ is
irrational. Let $f(x)$ be the minimal polynomial of $\lambda_1$. Then the
characteristic polynomial of $\mathbf{G}$ is divided by $f(x)^{n-d}$. This means that
any algebraic conjugate of $\lambda_1$ is also an eigenvalue of $\mathbf{G}$ with
multiplicity $n-d$. However, since $\Phi$ is a Levenstein-equality packing, we have $n>\frac{d(d+1)}{2}$. Combining with $d\geq4$, we have $d<n-d$, meaning that $\mathbf{G}$ does not have two eigenvalues with the same
multiplicity. This is a contradiction.  Hence, $\lambda_1$ is rational.  Since
$\lambda_2=-\lambda_1\cdot\frac{n-d}{d}$, we obtain that $\lambda_2$ is also
rational. This completes the proof.
\end{proof}

\begin{remark}\label{rm5.2}
 From viewpoint of the association scheme, the authors in \cite[Theorem
8.1]{Bannai09} shows that $\frac{1}{\alpha_{n,d}}$ is an integer if the strongly
regular graph generated by $\Phi$ in Lemma \ref{depend} is not a conference graph.
Recall that a strongly regular graph is a conference graph iff its parameters are
$(n, \frac{n-1}{2},\frac{n-5}{4},\frac{n-1}{4})$. A simple calculation shows that the
strongly regular graph with parameters described in (\ref{para2}) can never be a
conference graph. Hence, the results in \cite[Theorem 8.1]{Bannai09} also imply the
integrality of $\frac{1}{\alpha_{n,d}}$.
\end{remark}

Now we are ready to prove Theorem \ref{lev}.
\begin{proof}[Proof of Theorem \ref{lev}]
Recall that $\Phi\subset\Sd$ has the angle set
$A(\Phi)=\{0,\alpha_{n,d},-\alpha_{n,d}\}$, where $\alpha_{n,d}$ is defined in
(\ref{aleven}). {We claim that $n\geq \frac{d(d+3)}{2}$ if $n\neq \frac{d(d+2)}{2}$}.
According to Lemma \ref{al}, we can assume $\alpha_{n,d}=\frac{1}{k}$ for some
positive integer $k$.  Noting
$\alpha_{n,d}=\sqrt{\frac{3n-d(d+2)}{(d+2)(n-d)}}=\frac{1}{k}$, we have
\[
3k^2\cdot(n-\frac{d(d+2)}{3})=(d+2)(n-d)
\]
which implies
\begin{equation}\label{leven11}
(3k^2-d-2)(n-\frac{d(d+2)}{3})=\frac{1}{3}\cdot d(d-1)(d+2).
\end{equation}
We set $\alpha:=3k^2-d-2\in\Z$. Since either $n\geq \frac{d(d+3)}{2}$  or $n=
\frac{d(d+2)}{2}$, we have $n>\frac{d(d+2)}{3}$. Then (\ref{leven11}) implies that
$\alpha\in\Z_+$ and
\begin{equation}\label{eq:nalpha}
n=\frac{d(d+2)(d-1+\alpha)}{3\alpha}.
\end{equation}
We next show that $\alpha \in [2,\frac{2(d-1)(d+2)}{d+5}]\cap \Z$  if $n\neq
\frac{d(d+2)}{2}$. Indeed, {$\alpha \leq \frac{2(d-1)(d+2)}{d+5}$} follows from
$n\geq \frac{d(d+3)}{2}$. We still need show that $\alpha\geq 2$.
 For the aim of contradiction, we assume that $\alpha=1$.
  Then (\ref{eq:nalpha}) implies $n=\frac{d^2(d+2)}{3}$.
Noting that $\Phi\cup-\Phi$ is an antipodal $4$-distance set, we have
$\abs{\Phi\cup-\Phi}=2n=\frac{2d^2(d+2)}{3}$. On the other hand, according to
Delsarte-Goethals-Seidel bound, i.e. (\ref{absolute1}), we have
$\abs{\Phi\cup-\Phi}\leq \frac{d(d+1)(d+2)}{3}$. We have $\frac{2d^2(d+2)}{3}\leq
\frac{d(d+1)(d+2)}{3}$ , which is a contradiction.

We still need show $n\geq \frac{d(d+3)}{2}$ if $n\neq \frac{d(d+2)}{2}$. We assume
$n\neq \frac{d(d+2)}{2}$. According to Lemma \ref{depend}, $\Phi$ gives a strongly
regular graph $\Gamma$ with parameters described in (\ref{para2}). Let $E_2$ denote
the eigenspace of the adjacency matrix of $\Gamma$ with respect to the eigenvalue
$r_2$, and let $Y$ denote the spherical embedding of $\Gamma$ with respect to $E_2$.
{Since $Y$ is obtained by orthogonally projecting a standard basis of $\mathbb{R}^n$
onto the eigenspace $E_2$ and rescaling to have unit norm}, we know that
$Y\subset\mathbb{S}^{n_2-1}$. Combining with (\ref{nn}), we have
$Y\subset\mathbb{S}^{n-\frac{d(d+1)}{2}-1}$. By substituting (\ref{para2}) into
(\ref{embed}), we see that $Y$ is a spherical two-distance set with the angle set
$\{-\frac{d}{n-d},\frac{d}{2n-d(d+1)}\}$. {Since the Levenstein bound (\ref{leven})
is attained only if $n>\frac{d(d+1)}{2}$ \cite[Theorem 6.13]{leven2}, we have
$-\frac{d}{n-d}<0$. Also note that $\frac{d}{2n-d(d+1)}\neq 1$ since $n\neq
\frac{d(d+2)}{2}$. Hence, $Y$ contains no repeated vectors. According to the
Delsarte-Goethals-Seidel bound for spherical two-distance sets (see \cite[Theorem
4.8]{Delsarte}), we have
\begin{equation}\label{two}
|Y|=n\leq \frac{(n-\frac{d(d+1)}{2})(n-\frac{d(d+1)}{2}+3)}{2}.
\end{equation}
Rearranging the terms in (\ref{two}) and solving  a quadratic inequality gives $n\geq
\frac{d(d+3)}{2}$ or $n\leq\frac{(d-2)(d+1)}{2}$. Since $n>\frac{d(d+1)}{2}$, we
obtain $n\geq \frac{d(d+3)}{2}$. Hence, we have $n\geq \frac{d(d+3)}{2}$ if $n\neq
\frac{d(d+2)}{2}$. }
\end{proof}


\begin{thebibliography}{10}

\bibitem[BB09a]{Bannai}
Ei. Bannai, Et. Bannai,
\newblock {\em A survey on spherical designs and algebraic combinatorics on spheres,}
\newblock {\em European J. Combin.,}  30 (6) (August 2009), 1392-1425.

\bibitem[BB09b]{Bannai09}
Ei. Bannai, Et. Bannai,
\newblock {\em Spherical designs and Euclidean designs,}
\newblock {\em In: Recent Developments in Algebra and Related Areas, Beijing, 2007, Adv. Lect. Math., 8, Higher Education Press, Beijing; International Press, Boston,}  (2009), 1-37.

\bibitem[BD79]{Bannai2}
Ei. Bannai and R. Damerell,
\newblock {\em Tight spherical designs, I,}
\newblock {\em J. Math. Soc. Japan.,}  31 (1979), 199-207. MR0519045 (80b:05014)

\bibitem[BMV04]{Bannai3}
Ei. Bannai, A. Munemasa, B. Venkov,
\newblock {\em The nonexistence of certain tight spherical designs,}
\newblock {\em European J. Combin.,}  Algebra i Analiz 16:4 (2004), 1-23.

\bibitem[BBXYZ20]{Bannai4}
Ei. Bannai, Et. Bannai, Z. Xiang, W. Yu, Y. Zhu,
\newblock {\em Classification of Spherical $2$-distance $\{4,2,1\}$-designs by Solving Diophantine Equations,}
\newblock {\em Taiwanese J. Math.,}  advance publication, 25 June 2020. doi:10.11650/tjm/200601. https://projecteuclid.org/euclid.twjm/1593050477

\bibitem[BGMPV19]{Bilyk}
D. Bilyk, A. Glazyrin, R. Matzke, J. Park, O. Vlasiuk,
\newblock {\em Optimal measures for p-frame energies on spheres,}
\newblock {\em ArXiv preprint (2019), arXiv:1908.00885.}

\bibitem[BPR17]{Bondarenko}
A.V. Bondarenko, A. Prymak, D. Radchenko,
\newblock {\em Non-existence of $(76,30,8,14)$ strongly regular graph, }
\newblock {\em Linear Algebra Appl.,}  527 (2017), 53-72 http://dx.doi.org/10.1016/j.laa.2017.03.033

\bibitem[BH11]{Brouwer2}
A. E. Brouwer, W. H. Haemers,
\newblock {\em Spectra of graphs,}
\newblock { Springer Science \& Business Media,} 2011.

\bibitem[Cam04]{Cameron2}
P.J. Cameron,
\newblock {\em Strongly regular graphs,}
\newblock  {\em Topics in Algebraic Graph Theory,} 102 (2004), 203-221.

\bibitem[CHS96]{Conway}
J. H. Conway, R. H. Hardin, N. J. A. Sloane,
\newblock {\em Packing lines, planes, etc.: packings in Grassmannian spaces,}
\newblock  {\em Experiment. Math.,} 5(2), 139-159 (1996).

\bibitem[CK07]{Cohn}
H. Cohn and A. Kumar,
\newblock {\em Universally optimal distribution of points on spheres,}
\newblock  {\em J. Amer. Math. Soc.,} 20 (2007), no.1, 99-148.

\bibitem[DGS77]{Delsarte}
P. Delsarte, J. M. Goethals, and J. J. Seidel,
\newblock {\em Spherical codes and designs,}
\newblock  {\em Geometriae Dedicata} 6 (1977), 363-388.

\bibitem[FJM18]{Fickus2}
M. Fickus, J. Jasper, D. G. Mixon,
\newblock {\em Packings in real projective spaces,}
\newblock {\em SIAM J. Appl. Algebra Geometry,} 2(3), 377-409.

\bibitem[FM15]{ETF}
M. Fickus, D. G. Mixon,
\newblock {\em Tables of the existence of equiangular tight frames,}
\newblock {\em ArXiv preprint (2015), arXiv:1504.00253.}

\bibitem[HHM17]{Haas}
J. I. Haas, N. Hammen, D. G. Mixon,
\newblock {\em The Levenstein bound for packings in projective spaces, }
\newblock {\em Wavelets and Sparsity XVII,} Vol. 10394, p. 103940V (24 August 2017), International Society for Optics and Photonics. https://doi.org/10.1117/12.2275373

\bibitem[HS96]{Hardin}
R. H. Hardin, N. J. A. Sloane,
\newblock {\em McLaren's improved snub cube and other new spherical designs in three dimensions, }
\newblock {\em Discrete Comput. Geom.,} 15 (1996), 429-441. https://doi.org/10.1007/BF02711518

\bibitem[JKM19]{Jasper}
J. Jasper, E. J.  King, D. G. Mixon,
\newblock {\em  Game of Sloanes: best known packings in complex projective space,}
\newblock {\em Wavelets and Sparsity XVIII,}  vol. 11138, p. 111381E. International Society for Optics and Photonics, 2019.

\bibitem[LS73]{lemmens}
P. W. H. Lemmens, J. J. Seidel,
\newblock {\em Equiangular lines,}
\newblock {\em J. Algebra,} 24 (1973) 494-512.

\bibitem[Lev92]{leven}
V. I. Levenshtein,
\newblock {\em Designs as maximum codes in polynomial metric spaces,}
\newblock {\em Acta Appl. Math.,} 29 (1992), 1-82.

\bibitem[Lev98]{leven2}
V. I. Levenshtein,
\newblock {\em Universal bounds for codes and designs,}
\newblock {\em Handbook of coding theory,} 1 (1998), 499-648.

\bibitem[Mun07]{Munemasa}
A. Munemasa,
\newblock {\em  Spherical Designs, in: Handbook of Combinatorial Designs, 2nd ed., C. J. Colbourn, J. H. Dinitz, eds.}
\newblock {\em CRC Press,} (2007), 617-622.

\bibitem[Neu81]{Neumaier}
A. Neumaier,
\newblock {\em Combinatorial configurations in terms of distances,}
\newblock {\em Memorandum 81-09 (Dept. of Mathematics),} (1981), Eindhoven University of Technology.

\bibitem[NS11]{Nozaki}
H. Nozaki and S. Suda,
\newblock {\em Bounds on $s$-distance sets with strength $t$,}
\newblock {\em SIAM J. Discrete Math.,} 25(4), 1699-1713.

\bibitem[Sco73]{Scott}
L. L. Scott Jr,
\newblock {\em A condition on Higman's parameters,}
\newblock {\em Notices of the American Mathematical Society,}  20 (1973), A-97.

\bibitem[STDH07]{ETF2}
M. A. Sustik, J. A. Tropp, I. S. Dhillon, R. W. Heath,
\newblock {\em On the existence of equiangular tight frames,}
\newblock {\em Linear Algebra Appl.,} 426 (2007), 619-635.

\bibitem[Sze39]{Szego}
G. Szeg$\ddot{o}$,
\newblock {\em Orthogonal polynomials (Vol. 23),}
\newblock {\em American Mathematical Soc.,}  (1939).

\bibitem[Tay77]{Taylor}
D. E. Taylor,
\newblock {\em Regular 2-graphs,}
\newblock {\em Proc. London Math. Soc.,} (3), 35(2):257-274, 1977.

 \bibitem[Tre77]{Tremain}
J. C. Tremain,
\newblock {\em Concrete Constructions of Real Equiangular Line Sets,}
\newblock {\em ArXiv preprint (2008), arXiv:0811.2779.}

\bibitem[Wal09]{Waldron}
S. Waldron,
\newblock {\em On the construction of equiangular frames from graphs,}
\newblock {\em Linear Algebra Appl.,} 431 (2009), 2228-2242.

\bibitem[Wal16]{Waldron2}
S. Waldron,
\newblock {\em An introduction to finite tight frames,}
\newblock {\em \rm Birkh$\ddot{a}$user/Springer,} New York, 2016.

\bibitem[Wel74]{Welch}
L. Welch,
\newblock {\em Lower bounds on the maximum cross correlation of signals,}
\newblock {\em IEEE Trans. Inform. Theory.,} 20(3), 397-399, (1974).

\end{thebibliography}
\end{document}